%% file: epositive10Aug18arxiv.tex
\documentclass[11pt]{article}
\usepackage{amssymb}
\usepackage{amsthm}
\usepackage{latexsym}
\usepackage{tikz}
\usepackage{amsfonts}
\usepackage{tkz-graph}
\usepackage{amssymb}
\usepackage{comment}
\usepackage{tikz}
\usepackage{geometry}

\newtheorem{theorem}{Theorem}[section]

\newtheorem{corollary}[theorem]{Corollary}
\newtheorem{lemma}[theorem]{Lemma}
\newtheorem{definition}[theorem]{Definition}
\newtheorem{conjecture}[theorem]{Conjecture}

\newtheorem{case}{Case}
\newtheorem{case2}{Case}

\usepackage{color}

\def\inst#1{$^{#1}$}
\title{Classes of graphs with $e$-positive chromatic symmetric function}
\author{Ang\`ele M. Foley\inst{1}\footnote{Formerly Ang\`ele M. Hamel.} \and Ch\'inh T. Ho\`ang\inst{1} \and Owen D. Merkel \inst{1} }
\begin{document}
\maketitle
\begin{center}
{\footnotesize
\inst{1} Department of Physics and Computer Science, Wilfrid Laurier
University, \\Waterloo, Ontario, Canada }
\end{center}
\begin{abstract}
In the mid-1990s, Stanley and Stembridge conjectured that the chromatic symmetric functions of claw-free co-comparability (also called incomparability) graphs were $e$-positive. The quest for the proof of this conjecture has led to an examination of other, related graph classes. In 2013 Guay-Paquet proved that if unit interval graphs are $e$-positive, that implies claw-free incomparability graphs are as well. Inspired by this approach, we consider a related case and prove  that unit interval graphs whose complement is also a unit interval graph are $e$-positive.   We introduce the concept of strongly $e$-positive to denote a graph whose induced subgraphs are all $e$-positive, and conjecture that a graph is strongly $e$-positive if and only if it is (claw, net)-free.  
\end{abstract}

\section{Introduction}
\label{introduction}

A 1995 paper of Stanley \cite{Sold} introduced the chromatic symmetric functions and  proved a host of properties about them. A key element of this foundational paper was a  conjecture due to Stanley and Stembridge (originally stated in other terms in \cite{stanley}) that the chromatic symmetric functions of claw-free co-comparability (also called incomparability) graphs had the property known as $e$-positivity (defined in Section \ref{Background}).  As of this writing, this conjecture remains unproved, and work on it and on related results has fueled research in the area for over 20 years. 
A fundamental contribution to this endeavour was Guay-Paquet's result that if Stanley and Stembridge's conjecture holds for unit interval graphs, then it holds for claw-free co-comparability graphs \cite{GP}.  This result has put a spotlight on unit interval graphs. In related work, Shareshian and Wachs \cite{Shareshian} conjectured that co-comparability graphs of natural unit interval orders are $e$-positive, and Cho and Huh \cite{ChoHuh} and Harada and Precup \cite{HP} have proved $e$-positivity for several subclasses of unit interval graphs.  The time is ripe for further investigations of subclasses and superclasses of unit interval graphs.

Graphs and their complements are natural pairs to study.  The (claw, co-claw)-free graphs hold particular interest.  Two of the authors investigated them in \cite{hamel}, concluding they were not  all $e$-positive.  Here we revisit this result, showing that the particular graph called the net is the only exception.  
This result follows by careful consideration of the graph structure, and subsequent decomposition into constituent graphs.  From this analysis, along with a number of powerful graph theory results, we derive a series of results, culminating in a theorem that states that if a graph $G$ and its complement are both unit interval graphs, then $G$ is $e$-positive.

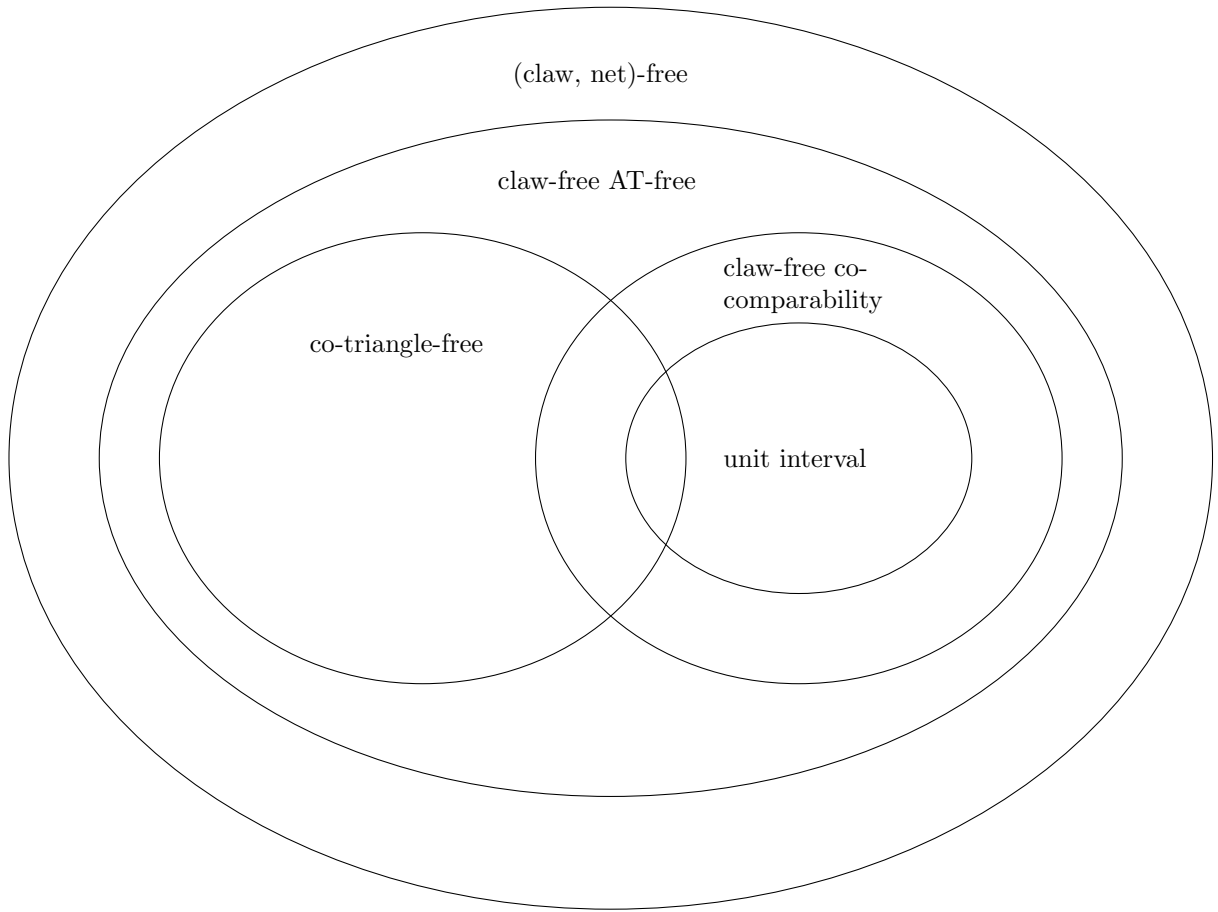
\begin{figure}
\begin{center}
\begin{tikzpicture} [scale = 1]
\tikzstyle{every node}=[font=\small]
\newcommand{\size}{1}
\newcommand{\Universe{1}}{

	\draw (0, 0) ellipse (8cm and 6cm);
	\draw (0, 0) ellipse (6.8cm and 4.5cm);
	\draw (-2.5, 0) ellipse (3.5cm and 3.0cm);
	\draw (2.5, 0) ellipse (3.5cm and 3.0cm);
	\draw (2.5, 0) ellipse (2.3cm and 1.8cm);

    	\node[text width=5cm] at (1.2,5.1) {(claw, net)-free};
   	\node[text width=5cm] at (1,3.7) {claw-free AT-free};
	\node[text width=3cm] at (3,2.3) {claw-free co-comparability};
	\node[text width=5cm] at (-1.5,1.5) {co-triangle-free};
	\node[text width=5cm] at (4,0) {unit interval};

}

\end{tikzpicture}
\end{center}
\caption{Classes of (claw, net)-free graphs. If the graphs are connected, the co-triangle class and the claw-free, co-comparability class actually partition the claw-free, AT-free class.}\label{fig:universe}
\end{figure}

The graph class universe we are working in is captured by Figure \ref{fig:universe}. The class of claw-free co-comparability graphs targeted by Stanley and Stembridge  wholly contains the subclass of unit interval graphs.  
If we look at the larger picture we see that the superclass of claw-free, AT-free graphs  (see definition of AT-free in Section \ref{Background}) consists of co-triangle-free graphs (known to be $e$-positive \cite{StanleyBook}, restated in Theorem \ref{cotriangleetal}) and claw-free co-comparability graphs. Thus proving the Stanley and Stembridge conjecture would prove all claw-free, AT-free graphs were $e$-positive. 

Even farther beyond this is the class of (claw, net)-free graphs. The net (see Figure \ref{fig:net})  is significant as this is the example originally given by Stanley \cite{Sold} of a claw-free, non-$e$-positive, graph to show claw-free alone is not a property sufficient to guarantee $e$-positivity.   We focus particularly on (claw, net)-free graphs (note that for $n=4$ there is one non-$e$-positive graph (namely, the claw, $K_{1,3}$), for $n=5$ there are 4 non-$e$-positive connected graphs (namely $K_{1,4}$, dart, cricket $= K_{1,4}+e$, co-$\{K_3\cup 2K_1\}$), for $n=6$, there are 44 non-$e$-positive connected graphs, and for $n=7$ there are 374 non-$e$-positive connected graphs).
To our knowledge, this paper is the first exploration of the (claw, net)-free $e$-positivity question.  We conjecture these graphs are $e$-positive. We have verified our conjecture for graphs up to $9$ vertices. We also introduce the term {\em strongly e-positive} to denote graphs whose induced subgraphs are also $e$-positive, and we conjecture a graph is strongly $e$-positive if and only if it is (claw, net)-free.

The paper is structured as follows. Section \ref{Background} covers background and notation from both graph theory and symmetric function theory. It also summarizes much of what is already known about which graphs are $e$-positive.  Section \ref{unitintervalgraphs} proves our result on the $e$-positivity of unit interval graphs whose complement is also a unit interval graph. Along the way we consider the $e$-positivity question for (claw, co-claw)-free graphs. Section \ref{strongly} contains our conjectures about strongly $e$-positive graphs and about (claw, net)-free graphs.

\section{Background and Notation}
\label{Background}

We begin by defining both graph theory and symmetric function terms and notation.
Let $G$ = ($V$,$E$) be a finite, simple, undirected graph with vertex set $V$ and edge set $E$.  We assume all graphs are connected, an assumption necessary because of Lemma \ref{lem:chromatic_disjoint_union} . For vertices $u,v \in V$, define $d(u,v)$ to be the length of the shortest path between $u$ and $v$. For a vertex $v \in V$, the open neighbourhood of $v$ is defined by $N(v) = \{u \in V : uv \in E\}$. For $U \subseteq V$, let $[U]$ denote the induced subgraph of $G$ induced by $U$. For a set $H$ of graphs, define $H$-free to be the class of graphs that do not contain any graph in $H$ as an induced subgraph.

\begin{figure}
\begin{center}
\begin{tikzpicture}[scale=1.25]
\tikzstyle{every node}=[font=\small]
\newcommand{\size}{1}
\newcommand{\net}{1}{
    \path (\size * 4, 0) coordinate (g1);
    \path (g1) +(-\size, 0) node (g1_1){};
    \path (g1) +(0, 0) node (g1_2){};
    \path (g1) +(\size, 0) node (g1_3){};
    \path (g1) +(2 * \size, 0) node (g1_4){};
    \path (g1) +(\size/2, 0.75^0.5) node (g1_5){};
    \path (g1) +(\size/2, 0.75^0.5+\size) node (g1_6){};
    \foreach \Point in {(g1_1), (g1_2), (g1_3), (g1_4), (g1_5), (g1_6)}{
   \node at \Point{\textbullet};
    }
    \draw (g1_1) -- (g1_2)
              (g1_2) -- (g1_3)
              (g1_3) -- (g1_4)
              (g1_5) -- (g1_2)
              (g1_5) -- (g1_3)
              (g1_5) -- (g1_6);
}
	\end{tikzpicture}
\end{center}
\caption{The net}\label{fig:net}	
\end{figure}

 Let $P_k$ be the chordless path on $k$ vertices and $C_k$ be the chordless cycle on $k$ vertices.  The complete graph (or clique) $K_n$ is the graph on $n$ vertices such that there is an edge between all pairs of vertices.  
A $K$-chain is a graph that is a sequence of complete graphs attached to one another sequentially at a single vertex, i.e.\ a vertex can belong to at most two maximal cliques.
The graph $K_3$ is called the {\em triangle}, and its complement $3K_1$, is called the {\em co-triangle}. 
The bull graph is the graph on 5 vertices and 5 edges arranged as a triangle with two pendant edges. See Figure \ref{fig:bull}. The generalized bull graphs are the family of graphs that can be constructed from the bull graph where each vertex in the triangle of the bull is substituted by a clique (nonempty). See Figure \ref{fig:generalizedbull}. 

\begin{figure}
\begin{center}
\begin{tikzpicture}[scale=1.25]
\tikzstyle{every node}=[font=\small]
\newcommand{\size}{1}
\newcommand{\bull}{1}{
\path (\size * 4, - \size * 4) coordinate (g8);
 \path (g8) +(0, 0) node (g8_1){};
\path (g8) +(0, \size) node (g8_2){};
\path (g8) +(\size, \size) node (g8_3){};
\path (g8) +(\size, 0) node (g8_4){};
\path (g8) +(\size/2, -0.75^0.5) node (g8_5){};
\foreach \Point in {(g8_1), (g8_2), (g8_3), (g8_4), (g8_5)}{ \node at \Point{\textbullet};}
\draw (g8_1) -- (g8_2)
(g8_3) -- (g8_4)
(g8_1) -- (g8_4)
(g8_5) -- (g8_4)
(g8_5) -- (g8_1);
}
\end{tikzpicture}
\end{center}
\caption{The bull}\label{fig:bull}	
\end{figure}
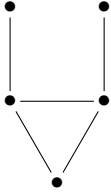

\begin{figure}
\begin{center}
\begin{tikzpicture} [scale = 1.25]
\tikzstyle{every node}=[font=\small]
\newcommand{\size}{1}
\newcommand{\GeneralizedBull}{1}{

    \path (0, 0) coordinate (g1);
    \path (g1) +(-1.5 * \size, 0) node (g1_1){}; 
    \path (g1) +(1.5 * \size, 0) node (g1_3){}; 
    \path (g1) +(-1.5 * \size, -1 * \size) node (g1_4){}; 
    \path (g1) +(1.5 * \size, -1 * \size) node (g1_5){}; 
    \path (g1) +(-1 * \size, 0) node (g1_6){}; 
    \path (g1) +(-1.1 * \size, -1.95 * \size) node (g1_7){}; 
    \path (g1) +(1.1 * \size, -1.95 * \size) node (g1_8){}; 
    \path (g1) +(1 * \size, 0) node (g1_9){}; 
    \path (g1) +(0, -0.5 * \size) node (g1_10){}; 
    \path (g1) +(-1 * \size, -1.94 * \size) node (g1_11){}; 
    \path (g1) +(1 * \size, -1.94 * \size) node (g1_12){}; 
    \path (g1) +(0, -2.5*\size) node (g1_13){};
    \foreach \Point in {(g1_1), (g1_3)}{
        \node at \Point{\textbullet};
    }
    \draw (-1.5 * \size, -1.5 * \size) ellipse (1cm and 0.5cm);
    \draw (1.5 * \size, -1.5 * \size) ellipse (1cm and 0.5cm);
    \draw (0, -3*\size) ellipse (1cm and 0.5cm);
    \draw   (g1_1) -- (g1_4.center)
            (g1_3) -- (g1_5.center)
            (g1_13.center) -- (g1_11.center)
            (g1_13.center) -- (g1_12.center)
            (g1_11.center) -- (g1_12.center);

}

\end{tikzpicture}
\end{center}
\caption{The generalized bull}\label{fig:generalizedbull}
\end{figure}
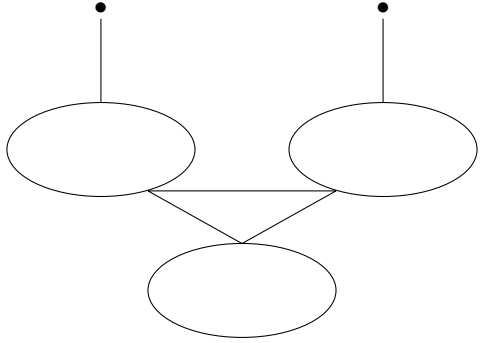

 A {\em stable set} is a set $S$ of vertices of a graph such that there are no edges between any of the vertices in $S$, e.g. a co-triangle is a stable set of size $3$. Let $\alpha(G)$ denote the size of the largest stable set in $G$. An \textit{astroidal triple (AT)} in a graph $G$ is a stable set of three vertices in $G$ such that for any pair of vertices  in the set, there is a path between them that does not intersect the neighbourhood of the third. A graph is called \textit{AT}-free exactly when it does not contain an astroidal triple. 

 A coloring of a graph $G$ is a function $\kappa$ from $V$ to the positive integers $\mathbb{Z}^+$: $\kappa:V\rightarrow \mathbb{Z}^+$. A coloring $\kappa$ is proper if $\kappa(u) \not= \kappa(v)$ whenever vertex $u$ is adjacent to vertex $v$. Chromatic symmetric functions were defined by Stanley \cite{Sold} as a generalization of the chromatic polynomial.  Indeed, if $x_1=x_2=x_3=\ldots =1$, this expression reduces to the chromatic polynomial for a graph.
\begin{definition}
For a graph $G$ with vertex set $V=\{v_1, v_2, \ldots , v_N\}$ and edge set $E$, the chromatic symmetric function is defined to be
\[
X_G=\sum_{\kappa} x_{\kappa(v_1)} x_{\kappa(v_2)}\cdots x_{\kappa(v_{N})}
\]
 where the sum is over all proper colorings $\kappa$ of $G$. 
\end{definition}

A function is symmetric if a permutation of the variables does not change the function. In precise terms, using the wording of Stanley \cite[p286]{StanleyBook}, ``it is a formal power series $\sum_{\alpha}c_{\alpha} x^{\alpha}$ where (a)  $\alpha$ ranges over all weak compositions $\alpha = (\alpha_1, \alpha_2, \ldots )$ of $n$ (of infinite length),  (b) $c_{\alpha} \in R$, (c) $x^{\alpha}$ stands for the monomial $x_{1}^{\alpha_{1}} x_{2}^{\alpha_{2}} \ldots$, and (d) $f(x_{w(1)}, x_{w(2)}, \ldots) = f(x_1, x_2, \ldots)$ for every permutation $w$ of the positive integers.'' Full background details can be found in Macdonald \cite{Macdonald} or Stanley \cite{StanleyBook}. 
It is well-known that certain sets of symmetric functions act as bases for the algebra of symmetric functions. One such set is the set of elementary symmetric functions.
The elementary symmetric function, $e_i(x)$, is defined as 
\[
e_i(x) = \sum_{j_1 < j_2< \cdots < j_i} x_{j_{1}} \cdots x_{j_{i}}.
\]
We can extend this definition using partitions. A partition $\lambda=(\lambda_1, \lambda_2, \ldots \lambda_{\ell})$ of a positive integer $n$ is a nondecreasing sequence of positive integers: $\lambda_1 \geq \lambda_2 \geq \ldots \geq \lambda_{\ell}$, where  $\lambda_i$ is called the $i$th part of $\lambda$, $1\leq i \leq \ell$.   The transpose, $\lambda'$, of $\lambda$, is defined by its parts: $\lambda_i' =  | \{j: \lambda_j \geq i\} |$. 
The elementary symmetric function, $e_\lambda(x)$, is defined as $e_\lambda(x) = e_{\lambda_{1}} e_{\lambda_{2}} \ldots e_{\lambda_{\ell}}$.

 If a given symmetric function can be written as a nonnegative linear combination of elementary symmetric functions we say the symmetric function is {\em e-positive}.  By abuse of notation we say a graph is $e$-positive if its chromatic symmetric function is $e$-positive. Furthermore, we say that a class of graphs is $e$-positive if every graph in the class is $e$-positive.

Note that the property of a graph being $e$-positive is not hereditary. That is, if a graph is $e$-positive, all of its induced subgraphs are not necessarily $e$-positive. 
For example, the chair  (or fork) graph (see Figure \ref{fig:chair}) is $e$--positive with chromatic symmetric function $X_F=  e_{2, 2, 1} + 2e_{3, 1, 1} + e_{3, 2} + 7e_{4, 1} + 5e_{5}$, but contains an induced claw $K_{1, 3}$ which is not $e$-positive, as $X_{K_{1,3}} = e_4+5e_{3,1} -2e_{2,2} + e_{2,1,1}$.  In Section \ref{strongly} we consider graphs whose induced subgraphs are all $e$-positive and dub these graphs {\em strongly $e$-positive}.

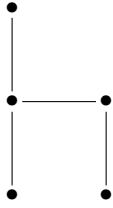
\begin{figure}
\begin{center}
\begin{tikzpicture}[scale=1.25]
\tikzstyle{every node}=[font=\small]
\newcommand{\size}{1}
\newcommand{\Chair}{1}{
    \path (\size * 4, - \size * 2) coordinate (g4);
    \path (g4) +(0, 0) node (g4_1){};
    \path (g4) +(0, \size) node (g4_2){};
    \path (g4) +(\size, \size) node (g4_3){};
    \path (g4) +(\size, 0) node (g4_4){};
    \path (g4) +(0, 2*\size) node (g4_5){};
    \foreach \Point in {(g4_1), (g4_2), (g4_3), (g4_4), (g4_5)}{
        \node at \Point{\textbullet};
    }
    \draw   (g4_1) -- (g4_2)
            (g4_2) -- (g4_3)
            (g4_3) -- (g4_4)
	   (g4_5) -- (g4_2);
}
\end{tikzpicture}
\end{center}
\caption{The chair}\label{fig:chair}	
\end{figure}

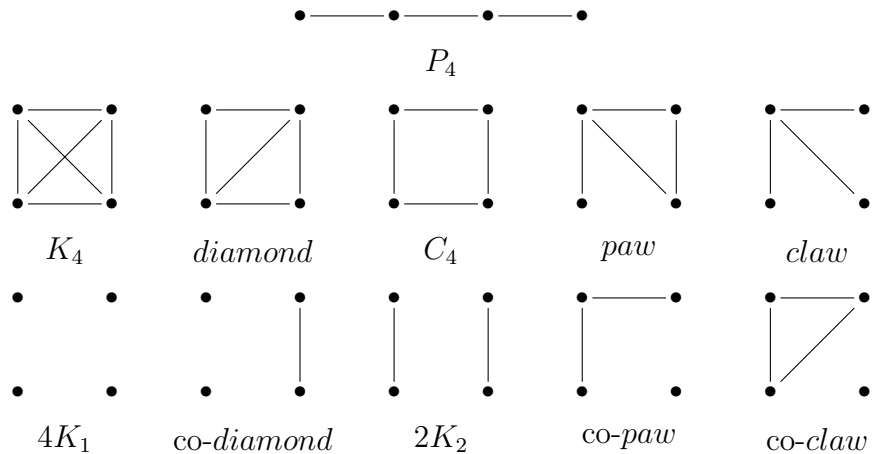
\begin{figure}
\input{All4VertexGraphs}
\caption{All four-vertex graphs.}
\end{figure}

\begin{figure}
\begin{center}
\begin{tabular}{||c|c|c||}
\hline \hline
Set $H$ & Positivity & Reference \\ 
\hline \hline
 $P_3$ & $e$-positive & Theorem \ref{cotriangleetalnew} \\ \hline
$3K_1$ & $e$-positive & \cite{StanleyBook} \\ \hline
claw, $K_3$ & $e$-positive & Theorem \ref{cotriangleetalnew}\\ \hline
claw, co-$P_3$ & $e$-positive & Theorem \ref{cotriangleetalnew} \\ \hline
\hline
\end{tabular}
\caption{Table of $e$-positivity for $H$-free graphs where $H$ contains a three-vertex graph}
\end{center}
\end{figure}

\begin{figure}
\label{tablecases}
\begin{center}
\begin{tabular}{||c|c|c||}
\hline \hline
Set  $H$ & Positivity & Reference\\
\hline \hline
claw, $P_4$ & $e$-positive &  \cite{Tsujie} \\  \hline
claw, paw & $e$-positive &  \cite{hamel} \\  \hline
claw, co-paw & $e$-positive &  \cite{hamel} \\  \hline
claw, co-claw (excluding the net) & $e$-positive  & Corollary \ref{e-co-claw} \\  \hline
claw, co-diamond & conjectured $e$-positive  & Conjecture \ref{mainconj} \\ \hline
claw, diamond  & not  necessarily $e$-positive & \cite{hamel} \\  \hline
claw, $K_4$ & not necessarily $e$-positive & \cite{hamel} \\  \hline
claw, $4K_1$ & not necessarily $e$-positive & \cite{hamel} \\  \hline
claw, $C_4$ & not necessarily $e$-positive & \cite{hamel} \\  \hline
claw, $2K_2$ & not necessarily $e$-positive & \cite{hamel} \\  
\hline \hline
\end{tabular}
\end{center}
\caption{Table of $e$-positivity for $H$-free graphs where $H$ contains four-vertex graphs.  Note that the exceptional case, the net, is not $e$-positive.}
\end{figure}

The following lemma  from Stanley \cite{Sold} is useful in constructing new classes of graphs:

\begin{lemma} [\cite{Sold}]
\label{lem:chromatic_disjoint_union}
If a graph $G$ is a disjoint union of subgraphs $G_1 \cup G_2$, then $X_G = X_{G_{1}}X_{G_{2}}$.
\end{lemma}

The following theorem summarizes the $e$-positive status of a number of graph classes:

\begin{theorem}
The following graph classes are known to be $e$-positive (proofs in the individual references given):
\begin{enumerate}
\item $P_k$ \cite{Sold}
\item $C_k$ \cite{Sold}
\item $K_n$ \cite{CV}
\item co-triangle-free \cite{StanleyBook}
\item $K$-chains \cite{gebhard}
\item generalized bull \cite{ChoHuh}
\item (claw, $P_4$-free) \cite{Tsujie}
\item (claw, paw)-free \cite{hamel}
\item (claw, co-paw)-free \cite{hamel}
\end{enumerate}
\label{cotriangleetal}
\end{theorem}

\begin{theorem}
The following graph classes are $e$-positive:
\begin{enumerate}
\item $P_3$-free
\item (claw, triangle)-free 
\item (claw, co-$P_3$)-free
\end{enumerate}
\label{cotriangleetalnew}
\end{theorem}

\begin{proof} 
We have the following arguments:
\begin{description}
\item[Item 1:] If $G$ is $P_3$-free then the components of $G$ are cliques. By \cite{CV}, restated in Theorem \ref{cotriangleetal}and Lemma \ref{lem:chromatic_disjoint_union}, $G$ is $e$-positive.
\item[Item 2:] If $G$ is (claw, triangle)-free, then each component of $G$ is a chordless path or cycle. Then together with \cite{CV}, restated in Theorem \ref{cotriangleetal}, Lemma \ref{lem:chromatic_disjoint_union}, and \cite{Sold}, restated in Theorem \ref{cotriangleetal}, $G$ is $e$-positive.
\item[Item 3:] The class of (claw, co-$P_3$)-free graphs is a subclass of (claw, co-paw)-free graphs which was shown to be $e$-positive in \cite{hamel}.
\end{description}
\end{proof}

Finally, this is Stanley and Stembridge's celebrated conjecture.  They expressed it in terms of incomparability graphs of (3+1)-free posets.

\begin{conjecture}[\cite{stanley}]
\label{stanleyconj}
A claw-free, co-comparability graph  is $e$-positive.
\end{conjecture}

In \cite{hamel} classes of (claw, $H$)-free graphs were analyzed, where $H$ is a graph on 4 vertices. However several classes of graphs, including (claw, diamond)-free graphs and (claw, co-claw)-free graphs are not necessarily $e$-positive, as was demonstrated using a counter-example. The counter-example that is used is a six-vertex graph called the net (Figure \ref{fig:net}). Here we extend the results of \cite{hamel} to remark that there are infinite families of non-$e$-positive graphs:

\begin{lemma}
There are infinitely many (claw, diamond)-free graphs, (claw, $C_4$)-free graphs, and (claw, $K_4$)-free graphs that are not $e$-positive. 
\end{lemma}

\begin{proof}
The family of triangle tower graphs, described in \cite{Dahlberg} are not $e$-positive, but are claw, $C_4$, and $K_4$ free. 
\end{proof}

As a related issue we comment that
there are several (most likely infinite)  (claw, $2K_2$)-free graphs, that are not $e$-positive.
To see this consider
the family of graphs obtained by attaching with an edge three independent vertices of degree 1 to distinct vertices of a clique are not $e$-positive but are claw and $2K_2$ free.

Returning to Stanley's singular counter-example---the net---we focus on this special graph.  In the next section we will show it is the only (claw, co-claw)-free graph that is not $e$-positive. We also note that the net contains an asteroidal triple and this causes us to turn our focus to AT-free graphs. In particular, we note this significant result of Kloks, Kratsch, and M{\"u}ller which shows that the claw-free, co-comparability graphs are one half of the set of (claw, AT)-free graphs.

\begin{theorem} [\cite{kloks}]
A connected graph $G$ is claw-free and AT-free if and only if at least one of the following holds:
\begin{enumerate}
	\item $G$ is a claw-free co-comparability graph.
	\item $G$ is co-triangle-free.
\end{enumerate}
\end{theorem}

Together with a result from \cite{StanleyBook} (restated in Theorem \ref{cotriangleetal}), Conjecture \ref{stanleyconj} would imply that the class of claw-free AT-free graphs is $e$-positive.

\section{Unit interval graphs}
\label{unitintervalgraphs}

An interval graph $G$ is a graph whose vertices can be represented by intervals on a straight line where two vertices in $G$ are adjacent if and only if their corresponding intervals intersect. A unit interval graph is an interval graph whose intervals are given by unit lengths. It has been shown in \cite{Lek1962} that interval graphs are exactly the class of chordal $AT$-free graphs. At the same time, unit interval graphs have been shown to be exactly the class of claw-free interval graphs \cite{unitinterval}. Guay-Paquet \cite{GP} has proved that Conjecture \ref{stanleyconj} can be reduced to the statement that the chromatic symmetric function of unit interval graphs are $e$-positive. So proving certain classes of unit interval graphs are $e$-positive will support this conjecture. This gives the motivation to study $H$-free unit interval graphs or equivalently (claw, $H$)-free $AT$-free chordal graphs.  

We investigate two different angles on the $e$-positivity question on unit interval graphs: 1) unit interval graphs that are $H$-free for $H$ a four vertex graph; and, 2) co-claw-free unit interval graphs.

\subsection{$H$-free unit interval graphs}

The table in Figure \ref{table:summary} summarizes what is known about $H$-free  unit interval graphs, where $H$ is a four vertex graph. This lemma from Hempel and Kratsch is required in what follows:

\begin{figure}
\begin{center}
\begin{tabular}{||c|c|c||}
\hline \hline
Graph $H$ & Positivity & Reference \\ \hline \hline
$P_4$ & $e$-positive &  \cite{Tsujie} \\ \hline
paw & $e$-positive &  \cite{hamel}\\ \hline
co-paw & $e$-positive & \cite{hamel} \\ \hline
co-claw & $e$-positive & Theorem \ref{coclawunit} \\ \hline
diamond & $e$-positive & Theorem \ref{diamondunit} \\ \hline
co-diamond & unknown & unknown \\ \hline
$K_4$ &  unknown & unknown \\ \hline
$4K_1$ & unknown & unknown  \\ \hline
$2K_2$ & unknown & unknown \\ \hline
\hline
\end{tabular}
\caption{Table of $e$-positivity results for $H$-free unit interval graphs where $H$ is a four-vertex graph.}
\label{table:summary}
\end{center}
\end{figure}

\begin{lemma}[\cite{hempel}]
\label{hempel}
Let $G$=($V$,$E$) be a claw-free, AT-free graph. Let $N_0={w}$, $N_1=N(w)$, $N_2, \ldots, N_i = \{x \in V | d(x,w) = i\}$. Then the following holds:
\begin{enumerate}
	\item $N_i$ is a clique for all $i=0,2,3,\ldots$
	\item $\alpha([N_1]) \leq 2$.
\end{enumerate}
\end{lemma}

\begin{theorem}
\label{diamondunit}
If $G$ is a diamond-free unit interval graph, then $G$ is $e$-positive.
\end{theorem}

\begin{proof}
Let $G=(V,E)$ be a diamond-free unit interval graph. From Lemma \ref{hempel}, fix $w \in V$ and define $N_0={w}$, $N_1=N(w)$, $L_2, \ldots, L_i = \{x \in V | d(x,w) = i\}$. Then $N_i$ is a clique for all $i \neq 1$. Since $G$ is diamond-free, $[N_1]$ must be $P_3$-free. Then since $\alpha([N_1]) \leq 2$, either $[N_1]$ is a complete graph or the disjoint union of two complete graphs.

\begin{case2}
$[N_1]$ is a complete graph.
\end{case2}

\begin{proof}
Every vertex in $N_2$ must have exactly one neighbour in $N_1$. If $y \in N_2$ has two neighbours $x_1, x_2 \in N_1$ then $\{w, x_1, x_2, y\}$ induces a diamond. Then every vertex in $N_2$ must be adjacent to the same vertex in $N_1$ say $x$, or $G$ will contain an induced $C_4$. For $i \geq 0$, this arguement can be continuely applied to $N_{i+2}$ since there is a vertex in $N_i$ that every vertex in $N_{i+1}$ is adjacent to. Therefore $G$ is a $K$-chain and by \cite{gebhard}, restated in Theorem \ref{cotriangleetal}, $G$ is $e$-positive. 
\end{proof}

\begin{case2}
$[N_1]$ is the disjoint union of two complete graphs.
\end{case2}

\begin{proof}
Note that no vertex in $N_2$ can have more than one neighbour in $N_1$. If a vertex in $N_2$ has two neighbours in the same component of $[N_1]$, then $G$ will contain an induced diamond. If a vertex in $N_2$ has a neighbour in each component of $[N_1]$, then $G$ will contain an induced $C_4$. Thus every vertex in $N_2$ has exactly one neighbour in $N_1$. Every vertex in $N_2$ has the same neighbour in $N_1$ or $G$ will contain an induced $C_4$ or $C_5$. Then one component of $[N_1]$  has no neighbours in $N_2$. For $i \geq 0$, this arguement can be continuely applied to $N_{i+2}$ since there is a vertex in $N_i$ that every vertex in $N_{i+1}$ is adjacent to. Therefore $G$ is a $K$-chain and by \cite{gebhard}, restated in Theorem \ref{cotriangleetal}, $G$ is $e$-positive.
\end{proof}

 We also remark that it can be shown that that a graph is a $K$-chain graph if and only if it is diamond-free unit interval graph
 \end{proof}

\begin{theorem}
If $G$ is a ($2K_2$, co-diamond)-free unit interval graph, then $G$ is $e$-positive.
\end{theorem}
\begin{proof}
From \cite{hamel}, it was determined that (claw, co-diamond, $2K_2$)-free graphs that are not known to be $e$-positive are the generalized pyramid graphs. Note that $G$ is a {\em generalized pyramid} if every oval is a clique and there are all edges between any two ovals (see Figure~\ref{fig:generalizedpyramid}).  Note that if all 3 ovals of $G$ are non-empty, then $G$ contains an astroidal triple \{a, b, c\}. Since unit interval graphs are $AT$-free, it must be the case that exactly one oval of $G$ is empty. Then $G$ is a generalized bull graph, and from \cite{ChoHuh}, restated in Theorem \ref{cotriangleetal}, $G$ is $e$-positive.
\end{proof}

\begin{figure}
\begin{center}
\begin{tikzpicture} [scale = 1.25]
\tikzstyle{every node}=[font=\small]

\newcommand{\size}{1}

\newcommand{\GeneralizedPyramid}{1}{
    \path (0, 0) coordinate (g1);
    \path (g1) +(-3 * \size, 0) node (g1_1){}; 
    \path (g1) +(0, -3 * \size) node (g1_2){}; 
    \path (g1) +(3 * \size, 0) node (g1_3){}; 
    \path (g1) +(-2 * \size, -1.05 * \size) node (g1_4){}; 
    \path (g1) +(2 * \size, -1.05 * \size) node (g1_5){}; 
    \path (g1) +(-1 * \size, 0) node (g1_6){}; 
    \path (g1) +(-1.1 * \size, -1.95 * \size) node (g1_7){}; 
    \path (g1) +(1.1 * \size, -1.95 * \size) node (g1_8){}; 
    \path (g1) +(1 * \size, 0) node (g1_9){}; 
    \path (g1) +(0, -0.5 * \size) node (g1_10){}; 
    \path (g1) +(-1 * \size, -1.07 * \size) node (g1_11){}; 
    \path (g1) +(1 * \size, -1.07 * \size) node (g1_12){}; 
    \foreach \Point in {(g1_1), (g1_2), (g1_3)}{
        \node at \Point{\textbullet};
    }
    \draw (-1.5 * \size, -1.5 * \size) ellipse (1cm and 0.5cm);
    \draw (1.5 * \size, -1.5 * \size) ellipse (1cm and 0.5cm);
    \draw (0, 0) ellipse (1cm and 0.5cm);
    \draw   (g1_1) -- (g1_4.center)
            (g1_1) -- (g1_6.center)
            (g1_2) -- (g1_7.center)
            (g1_2) -- (g1_8.center)
            (g1_3) -- (g1_5.center)
            (g1_3) -- (g1_9.center)
            (g1_10.center) -- (g1_11.center)
            (g1_10.center) -- (g1_12.center)
            (g1_11.center) -- (g1_12.center);

    \node[text width=1cm] at (-3,0.2) {$a$};
    \node[text width=1cm] at (3.7,0.2) {$b$};
    \node[text width=1cm] at (0,-3) {$c$};
}

\end{tikzpicture}
\end{center}
\caption{A generalized pyramid}
\label{fig:generalizedpyramid}
\end{figure}
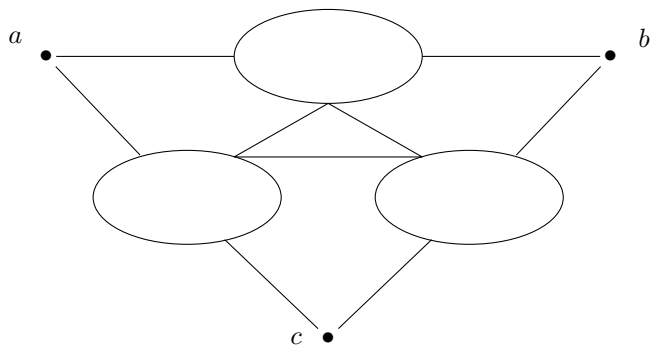

As a side issue we can consider the families of $2K_2$-free unit interval graphs that are not known to be $e$-positive. Let $G=(V,E)$ be a $2K_2$-free unit interval graph. We can assume $G$ is connected since if not $G$ has at most one component that is not an isolated vertex. From Lemma \ref{hempel}, fix $w \in V$ and define $N_0={w}$, $N_1=N(w)$, $L_2, \ldots, L_i = \{x \in V | d(x,w) = i\}$. Since $G$ is $2K_2$-free, $N_i = \emptyset$ for all $i \geq 4$. Then $N_i$ is a clique for $i = 0, 2, 3$ and $\alpha([N_1]) \leq 2$. If $N_3$ has an edge $\{z_1z_2\}$ then for any vertex $x \in N_1$, $\{w, x, z_1, z_2\}$ induces a $2K_2$ in $G$. Then either $N_3$ is empty or $N_3$ has a single vertex.

\begin{case}
Suppose $[N_1]$ is not connected.
\end{case}
\noindent Then $[N_1]$ has two components, say $N_{1,0}$ and $N_{1,1}$. At least one component of $[N_1]$ is a single vertex, say $N_{1,1}$, and the other is a clique. If $N_2$ is empty, then $G$ can be partitioned into 2 cliques, $N_0 \cup N_{1,0}$ and $N_{1,1}$, and by \cite{CV}, restated in Theorem \ref{cotriangleetal}, $G$ is $e$-positive. Now suppose $y_1 \in N_2$. Then vertices in $N_2$ can only have neighbours in one component of $[N_1]$ or $G$ will have a chordless cycle on at least 4 vertices. Then if $N_3 \neq \emptyset$ or $[N_2]$ has an edge, this edge together with the edge formed by $w$ and a vertex from the component of $[N_1]$ with no neighbours in $N_2$, will induce a $2K_2$ in $G$. Then $N_2$ has exactly one vertex and $N_3$ is empty. In this case $G$ is a generalized bull, and by  \cite{ChoHuh}, restated in Theorem \ref{cotriangleetal}, $G$ is $e$-positive.

\begin{case}
Suppose $[N_1]$ is connected.
\end{case}
\noindent Suppose $N_3$ has a single vertex, $z$. Any vertex in $y_1 \in N_2$ that is adjacent to $z$ must be adjacent to every vertex in $N_1$. If $y_1$ is not adjacent to $x \in N_1$, then $\{w, x, y_1, z\}$ induces a $2K_2$ in $G$. Also any vertex $y_2 \in N_2$ that is not adjacent to $z$ must be adjacent to every vertex in $N_1$. If $y_2$ is not adjacent to $x \in N_1$ then $\{y_1, y_2, z, x\}$ induces a claw in $G$. It must also be the case that $N_1$ is a clique. If $x_1,x_2 \in N_1$ are not adjacent then $\{x_1, x_2, y, z\}$ induces a claw in $G$. Then $G$ is a generalized bull and by  \cite{ChoHuh}, restated in Theorem \ref{cotriangleetal}, $G$ is $e$-positive.
\\\\\noindent Now suppose $N_3$ is empty. Then if $N_1$ is a clique, $G$ can be partitioned into two cliques, $N_0 \cup N_1$ and $N_2$, and by \cite{CV}, restated in Theorem \ref{cotriangleetal}, $G$ is $e$-positive. Then $N_1$ must contain an induced $P_3$. 

The family of $2K_2$-free unit interval graphs that are not known to be $e$-positive have $[N_1]$ connected, $[N_1]$ contains an induced $P_3$, $\alpha([N_1]) = 2$, $N_2 \neq \emptyset$, and all $N_i = \emptyset$ for $i \geq 3$.

\subsection{The structure of (claw, co-claw)-free graphs}

As a preliminary to considering co-claw, unit interval graphs, we investigate (claw, co-claw)-free graphs.
 First observe that the net is a (claw, co-claw)-free graph that is not $e$-positive. Here we will show this is the only graph in this class that is not $e$-positive.
 Theorems \ref{thm:claw-co-claw} and \ref{thm:net-or-3-sun} below are implicitly implied by the proof of Theorem 3 in \cite{HR}. For the sake of completeness, we will give a proof of both theorems here.  Note that the complement of the net is called the 3-sun.

\begin{theorem}[\cite{HR}]\label{thm:claw-co-claw}
Let $G$ be a (claw, co-claw)-free graph. If $G$ contains a triangle and a co-triangle that are vertex-disjoint, then $G$ contains a net or a 3-sun as an induced subgraph.
\end{theorem}

\begin{proof}
Let $G$ be a (claw, co-claw)-free graph. Suppose $G$ contains a triangle $T$ with vertices $c_1, c_2, c_3$ and a co-triangle $C$ with vertices $s_1, s_2, s_3$ such that the triangle and co-triangle are vertex-disjoint. We claim that 
\begin{equation}
\begin{minipage}{0.9\linewidth}
\label{eq:one-neighbor-in-T} For a triangle $R$, every vertex in $G-R$ is adjacent to at least one vertex in $R$.
\end{minipage}\end{equation} 
If (\ref{eq:one-neighbor-in-T}) failed then $R$ and a vertex of $G-C$ with no neighbors in $R$ would form a co-claw, a contradiction. 

Suppose $s_1$ is adjacent to all vertices of the triangle $T$. Consider the triangle $\{s_1, c_1, c_2\}$. By (\ref{eq:one-neighbor-in-T}), vertex $s_2$ is adjacent to $c_1$, or $c_2$. We may assume $s_2$ is adjacent to $c_2$. We have the following implications.\\
\hspace*{2em}  $c_2$ is not adjacent to $s_3$, for otherwise $\{c_2, s_1, s_2, s_3\}$ induces a claw, a contradiction;\\
\hspace*{2em}  $s_3$ is adjacent to $c_1$, for otherwise $\{s_3, c_1, c_2, s_1\}$ induces a co-claw, a contradiction;\\
\hspace*{2em}  $s_3$ is adjacent to $c_3$, for otherwise $\{s_3, c_2, c_3, s_1\}$ induces a co-claw, a contradiction;\\
\hspace*{2em}  $s_2$ is not adjacent to $c_1$, for otherwise $\{c_1, s_1, s_2, s_3\}$ induces a claw, a contradiction;\\
\hspace*{2em}  $s_2$ is not adjacent to $c_3$, for otherwise $\{c_3, s_1, s_2, s_3\}$ induces a claw, a contradiction;\\
\hspace*{2em}  $\{s_2, c_1, c_3, s_1\}$ induces a co-claw, a contradiction.

Thus, we may assume that every $s_i$ ($i = 1,2,3$) is non-adjacent to at least one $c_j$ ($ j = 1,2,3$). 
Consider the vertex $s_1$ and suppose $s_1$ is adjacent to two vertices of $\{c_1, c_2, c_3\}$.
We may assume $s_1$ is adjacent to $c_1, c_2$ and non-adjacent to $c_3$. 
If $c_1$ is adjacent to both $s_2, s_3$, then $\{c_1, s_1, s_2, s_3\}$ induces a claw, a contradiction. Thus, we may assume $c_1$ is not adjacent to $s_2$.  We have the following implications.\\
\hspace*{2em}  $s_2$ is adjacent to $c_2$, for otherwise, $\{s_2, c_1, c_2, s_1\}$ induces a co-claw;\\
\hspace*{2em} $s_3$ is not adjacent to $c_2$, for otherwise, $\{c_2, s_1, s_2, s_3\}$ induces a claw;\\
\hspace*{2em} $s_3$ is adjacent to $c_1$, for otherwise $\{s_3, c_1, c_2, s_1\}$ induces a co-claw;\\
\hspace*{2em} $s_3$ is adjacent to $c_3$, for otherwise $\{c_1, s_1, c_3, s_3\}$ induces a claw; \\
\hspace*{2em} $s_2$ is adjacent to $c_3$, for otherwise $\{s_2, c_1, c_3, s_3\}$ induces a co-claw.\\

Now, $T \cup C$ induces a 3-sun in $G$. So, we may assume every vertex in $C$ is adjacent to exactly one vertex in $T$. It is now easily to see that $T \cup C$ induces a net in $G$. 
	
\end{proof}

\begin{theorem}[\cite{HR}]\label{thm:net-or-3-sun}
Let $G$ be a (claw, co-claw)-free graph. If $G$ contains a net or a 3-sun as an induced subgraph, then $G$ is a net, or a 3-sun.
\end{theorem}
\begin{proof}
Let $G$ be a (claw, co-claw)-free graph. Suppose that $G$ contains a net with vertices $c_1, c_2, c_3, s_1, s_2, s_3$ 
	such that $C = \{c_1, c_2, c_3\}$ induces a triangle, $S = \{s_1, s_2, s_3\}$ induces a co-triangle, and $c_i$ is adjacent to $s_i$ for $i = 1,2,3$. 
	Consider a vertex $t$ not belonging to the net. Vertex $t$ cannot be adjacent to all vertices of $S$, for otherwise $S \cup \{t\}$ induces a claw. So we may assume $t$ is not adjacent to $s_1$. Suppose $t$ is adjacent to $c_1$. Since $\{c_1, s_1, c_2, t\}$ cannot induce a claw, $t$ must be adjacent to $c_2$. Similarly, since $\{c_1, s_1, c_3, t\}$ cannot induce a claw, $t$ must be adjacent to $c_3$. Now $\{s_1, t, c_2, c_3\}$ induces a co-claw, a contradiction. 
	
	So we know $t$ is not adjacent to $c_1$. Vertex $t$ must be adjacent to $c_2$, or $c_3$ (or both), for otherwise, $\{t, c_1, c_2, c_3\}$  induces a co-claw, a contradiction. Without loss of generality, assume $t$ is adjacent to $c_2$. Since $\{c_2, c_1, s_2, t\}$ cannot induce a claw, $t$ must be adjacent to $s_2$. But now $\{s_1, t, c_2, s_2\}$ induces a co-claw. We now can conclude that if $G$ contains a net, the $G$ is a net. By considering the complement of $G$, it follows that if $G$ contains a 3-sun, then $G$ is a 3-sun. 
\end{proof}

\begin{theorem}\label{thm:antenna}
Let $G$ be a (claw, co-claw)-free graph. If $G$ contains an antenna, then $G$ is an induced subgraph of the graph $F_1$.
\end{theorem}

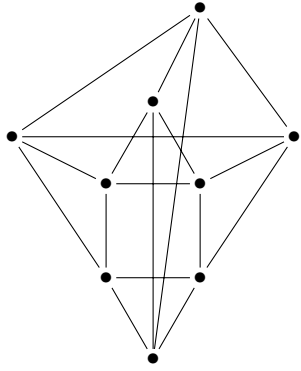
\begin{figure}
\begin{center}
\begin{tikzpicture}[scale=1.25]
\tikzstyle{every node}=[font=\small]
\newcommand{\size}{1}
\newcommand{\F1}{1}{
    \path (\size * 4, - \size * 2) coordinate (g4);
    \path (g4) +(0, 0) node (g4_1){};
    \path (g4) +(0, \size) node (g4_2){};
    \path (g4) +(\size, \size) node (g4_3){};
    \path (g4) +(\size, 0) node (g4_4){};
    \path (g4) +(\size/2, 0.75^0.5+\size) node (g4_5){};
    \path (g4) +(\size/2+0.5, 0.75^0.5+\size+\size) node (g4_6){};
   \path (g4) +(-\size, \size+0.5) node (g4_7){};
   \path (g4) +(2*\size, \size+0.5) node (g4_8){};
   \path (g4) +(\size/2, -0.75^0.5) node (g4_9){};
    \foreach \Point in {(g4_1), (g4_2), (g4_3), (g4_4), (g4_5), (g4_6), (g4_7), (g4_8), (g4_9)}{
        \node at \Point{\textbullet};
    }
    \draw   (g4_1) -- (g4_2)
            (g4_1) -- (g4_4)
            (g4_2) -- (g4_3)
            (g4_3) -- (g4_4)
	   (g4_5) -- (g4_2)
	   (g4_5) -- (g4_3)
            (g4_6) -- (g4_5)

            (g4_9) -- (g4_1)
            (g4_9) -- (g4_4)
            (g4_9) -- (g4_5)
            (g4_9) -- (g4_6)

            (g4_7) -- (g4_6)
            (g4_7) -- (g4_2)
            (g4_7) -- (g4_8)
            (g4_7) -- (g4_1)

           (g4_8) -- (g4_3)
           (g4_8) -- (g4_4)
           (g4_8) -- (g4_6);
}
\end{tikzpicture}
\end{center}
\caption{The graph $F_1$}\label{fig:F1}	
\end{figure}

\begin{figure}
\begin{center}
\begin{tikzpicture}[scale=1.25]
\tikzstyle{every node}=[font=\small]
\newcommand{\size}{1}
\newcommand{\antenna}{1}{
    \path (\size * 4, - \size * 2) coordinate (g4);
    \path (g4) +(0, 0) node (g4_1){};
    \path (g4) +(0, \size) node (g4_2){};
    \path (g4) +(\size, \size) node (g4_3){};
    \path (g4) +(\size, 0) node (g4_4){};
    \path (g4) +(\size/2, 0.75^0.5+\size) node (g4_5){};
    \path (g4) +(\size/2, 0.75^0.5+\size+\size) node (g4_6){};
    \foreach \Point in {(g4_1), (g4_2), (g4_3), (g4_4), (g4_5), (g4_6)}{
        \node at \Point{\textbullet};
    }
    \draw   (g4_1) -- (g4_2)
            (g4_1) -- (g4_4)
            (g4_2) -- (g4_3)
            (g4_3) -- (g4_4)
	   (g4_5) -- (g4_2)
	   (g4_5) -- (g4_3)
            (g4_6) -- (g4_5);
}
	\end{tikzpicture}
	\end{center}
	\caption{The antenna}\label{fig:antenna}	
\end{figure}
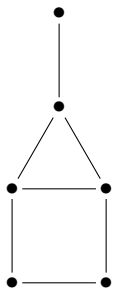

\begin{figure}
\begin{center}
\begin{tikzpicture}[scale=1.25]
\tikzstyle{every node}=[font=\small]
\newcommand{\size}{1}
\newcommand{\bull}{1}{
\path (\size * 4, - \size * 4) coordinate (g8);
 \path (g8) +(-1.5*\size, 0) node (g8_1){};
\path (g8) +(-0.5*\size, 0) node (g8_2){};
\path (g8) +(0.5*\size,0) node (g8_3){};
\path (g8) +(1.5*\size, 0) node (g8_4){};
\path (g8) +(0, -0.75^0.5) node (g8_5){};
\path (g8) +(-0.5*\size, 1.5*\size) node (g8_6){};
\path (g8) +(0.5*\size, -1.5\size) node (g8_7){};
\foreach \Point in {(g8_1), (g8_2), (g8_3), (g8_4), (g8_5), (g8_6), (g8_7)}{ \node at \Point{\textbullet};}
\draw (g8_1) -- (g8_2)
(g8_3) -- (g8_4)
(g8_1) -- (g8_4)
(g8_5) -- (g8_2)
(g8_5) -- (g8_3)

(g8_6) -- (g8_1)
(g8_6) -- (g8_2)
(g8_6) -- (g8_4)

(g8_7) -- (g8_1)
(g8_7) -- (g8_3)
(g8_7) -- (g8_4);
}
\end{tikzpicture}
\end{center}
\caption{The graph $F_2$}\label{fig:F2}	
\end{figure}
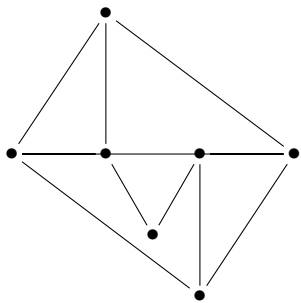

\begin{proof}
Let $G$ be a (claw, co-claw)-free graph. Suppose $G$ contains an antenna $A$ with vertices $c_1,c_2,c_3,s_1,s_2,s_3$ as indicated by Figure~\ref{fig:antenna}. Let $D_1$ be the set of vertices in $G -A$ that are adjacent to $s_2, s_3, c_1, s_1$ and no other vertices in $A$. Let $D_2$ be the set of vertices in $G -A$ that are adjacent to $ s_1, c_2,s_2$ and no other vertices in $A$.  Let $D_3$ be the set of vertices in $G -A$ that are adjacent to $ s_1, c_3,s_3$ and no other vertices in $A$. We claim that 

\begin{equation}
\begin{minipage}{0.9\linewidth}
\label{eq:D} Every vertex in $G-A$ belong to $D_1 \cup D_2 \cup D_3$. 
\end{minipage}\end{equation} 	

Let $v$ be a vertex in $G-A$. Consider the triangle $T$ with vertices $c_1, c_2,c_3$. Then $v$ must be adjacent to at least one vertex in $T$, for otherwise $v$ and $T$ form a co-triangle. Suppose that $v$ is adjacent to both $c_2$ and $c_3$.  We have the following implications:\newline
\hspace*{3em} $v$ is adjacent to $s_1$, for otherwise, $\{s_1, v, c_2, c_3\}$ induces a co-claw; \\
\hspace*{3em} $v$ is not adjacent to $s_3$, for otherwise, $\{v, s_1, c_2, s_3\}$ induces a claw;\\
\hspace*{3em} $v$ is adjacent to $c_1$, for otherwise, $\{c_3, c_1, v, s_3\}$ induces a claw;\\
\hspace*{3em} $\{s_3, v, c_1, s_1\}$ induces a co-claw, a contradiction.

So $v$ must be non-adjacent to $c_2$, or $c_3$. Suppose that $v$ is non-adjacent to both $c_2$ and $c_3$. We will show that $v$ must be in $D_1$. Note that $v$ is adjacent to $c_1$. We have the following implications:\newline
\hspace*{3em} $v$ is adjacent to $s_1$, for otherwise $\{c_1, s_1, v, c_3\}$ induces a claw; \\
\hspace*{3em} $v$ is adjacent to $s_3$, for otherwise $\{s_3, v, c_1, s_1\}$ induces a co-claw; \\
\hspace*{3em} $v$ is adjacent to $s_2$, for otherwise $\{s_2, v, c_1, s_1\}$ induces a co-claw; \\

Thus, $v$ belongs to $D_1$. So we may assume $v$ is adjacent to exactly one vertex of $\{c_2, c_3\}$. Suppose that $v$ is adjacent to $c_3$ but not to $c_2$. We will show that $v$ belongs to $D_3$. We have the following implications. \newline
\hspace*{3em} $v$ is adjacent to $s_3$, for otherwise, $\{ c_3, c_2, v, s_3\}$ induces a claw;\\
\hspace*{3em} $v$ is adjacent to $s_1$, for otherwise, $\{s_1, v, c_3, s_3\}$ induces a co-claw;\\
\hspace*{3em} $v$ is non-adjacent to $s_2$, for otherwise, $\{v, s_1,c_3,s_2\}$ induces a claw;\\
\hspace*{3em} $v$ is non-adjacent to $c_1$, for otherwise, $\{s_2, v, c_1,s_1\}$ induces a co-claw;

Thus, $v$ belongs to $D_3$. By symmetry, if  $v$ is adjacent to $c_2$ but not to $c_3$, then $v$ belongs to $D_2$. We have established (\ref{eq:D}). If some $D_i$ contains at least two vertices, then it is easy to see that $G$ contains a claw, or co-claw. Thus $G$ has at most 9 vertices and is an induced subgraph of $F_1$.
\end{proof}

\begin{theorem}\label{thm:bull}
Let $G$ be a (claw, co-claw)-free graph. If $G$ does not contain a net, a 3-sun, or an antenna, and contains a bull, then $G$ is an induced subgraph of the graph $F_2$.
\end{theorem}
\begin{proof}	
Let $G$ be a (claw, co-claw)-free graph that does not contain a net, or an antenna, but contains a bull $B$. Let the vertices of the bull $B$ be $c_1, c_2, c_3, s_1, s_2$ as indicated by Figure~\ref{fig:bull}. We may assume $G$ is not the bull, for otherwise we are done. Let $X$ be the set of vertices of $G-B$ that are adjacent to $s_1, c_1, s_2$ and non-adjacent to $c_2, c_3$. Let $Y$ be the set of vertices of $G-B$ that are adjacent to $s_1, c_2, s_2$ and non-adjacent to $c_1, c_3$. We are going to show that every vertex in $G-B$ belongs to $X \cup Y$. 

Consider a vertex $v \in G-B$. Suppose that $v$ is adjacent to $c_3$. Then $v$ is non-adjacent to at least one vertex of $\{s_1, s_2\}$, for otherwise, $\{v, s_1, s_2, c_3\}$ induces a claw, a contradiction. Without loss of generality, assume $v$ is non-adjacent to $s_1$. Then $v$ is non-adjacent to $c_2$, for otherwise, $\{s_1, v, c_2,c_3\}$ induces a co-claw. Now, $v$ is non-adjacent to $c_1$, for otherwise, $\{c_1, s_1,v,c_2\}$ induces a claw. But then the set $\{c_1, c_2, c_3, s_1, s_2, v\}$ induces a net or an antenna in $G$, a contradiction. So, $v$ is non-adjacent to $c_3$.

Consider the triangle with vertices $c_1, c_2, c_3$. Vertex $v$ must be adjacent to $c_1$, or $c_2$, for otherwise, the triangle and $v$ induce a co-claw. Suppose that $v$ is adjacent to $c_1$. Now $v$ must be adjacent to $s_1$, for otherwise, $\{c_1, s_1,v,c_3\}$ induces a claw. Then $v$ must be adjacent to $s_2$, for otherwise, $\{s_2, s_1, c_1, v\}$ induces a co-claw. Vertex $v$ is non-adjacent to $c_2$, for otherwise, $B$ and $v$ induce a 3-sun, and we are done. Now we know $v \in X$. Similarly, if $v$ is adjacent to $c_2$, the $v \in Y$. If $X$ or $Y$ contains two vertices, then it is easy to see $G$ contains a claw, or co-claw. So, $G$ has at most seven vertices and  is an induced subgraph of the graph $F_2$.
\end{proof}

Let ${\cal F}_1$ (respectively, ${\cal F}_2$) be the class of graphs $G$ such that $G$ or $\overline{G}$ contains an antenna (respectively, bull) and is an induced subgraph of $F_1$ (respectively $F_2$). 

\begin{theorem}\label{thm:main}
Let $G$ be a (claw, co-claw)-free graph. Then one of the following holds.
\begin{description}
	\item [(i)] $G$ or $\overline{G}$ contain no triangle.
	\item [(ii)] $G$ or $\overline{G}$ is the net.
	\item [(iii)] $G$ or $\overline{G}$ belong to ${\cal F}_1$
	\item [(iv)] $G$ or $\overline{G}$ belong to ${\cal F}_2$	
	\end{description}
\end{theorem}

\begin{proof}
 Let $G$ be a (claw, co-claw)-free graph. We may assume $G$ contains a triangle $T$ and a co-triangle $C$, for otherwise,  (i)  holds, and we are done.  The sets $T$ and $C$ cannot be vertex-disjoint, for otherwise, (ii) holds by Theorems~\ref{thm:claw-co-claw} and~\ref{thm:net-or-3-sun}, and we are done. So,  $T$ and $C$ must intersect at one vertex. It is easy to verify that $T \cup C$ induces a bull. We may assume $G$ contains no antenna, for otherwise, by Theorem~\ref{thm:antenna}, (iii) holds and we are done. By Theorem~\ref{thm:bull}, $G$ belongs to ${\cal F}_2$, and so (iv) holds, and we are done.
\end{proof}

\begin{corollary}
\label{e-co-claw}
If $G$ is a (claw, co-claw)-free graph that is not isomorphic to the net, then $G$ is $e$-positive.
\end{corollary}

\begin{proof}
This follows directly from the Theorem \ref{thm:main} and the facts that all graphs in ${\cal F}_1$ and ${\cal F}_2$ were verified by computer program to be $e$-positive, (claw, triangle)-free graphs and co-triangle-free graphs are $e$-positive, and the complement of the net is $e$-positive.
\end{proof}

The following theorem is one of our main results:

\begin{theorem}
\label{coclawunit}
If $G$ is a co-claw-free unit interval graph, then $G$ is $e$-positive.
\end{theorem}

\begin{proof}
The only (claw, co-claw)-free graph that is not $e$-positive is the net and the net is not a unit interval graph (since it contains an astroidal triple). Then it follows from Corollary \ref{e-co-claw} that co-claw-free unit interval graphs are $e$-positive.
\end{proof}

\begin{corollary}
If $G$ is a unit interval graph and the complement of $G$ is a unit interval graph, then $G$ is $e$-positive.
\end{corollary}
\begin{proof}
This follows directly from the previous theorem and the fact that if the complement of $G$ is a unit interval graph, then $G$ is co-claw-free.
\end{proof}

\section{Strongly $e$-positive graphs}
\label{strongly}

As discussed in Section \ref{introduction}, the property of a graph being $e$-positive is not hereditary. This gives motivation to seek out which graphs have this special property.   We coin the term {\em strongly e-positive} graphs for graphs with this property.

\begin{definition}
A graph $G$ is strongly $e$-positive if for all induced subgraphs $H$ of $G$, $H$ is $e$-positive.
\end{definition}

Note that the classes of claw-free co-comparability graphs, unit interval graphs, and (claw, co-diamond)-free graphs are all subclasses of (claw, net)-free graphs.  See Figure \ref{fig:universe}. We conjecture that this class of graphs is exactly the class of strongly $e$-positive graphs.

\begin{conjecture}
A graph is strongly $e$-positive if and only if it is (claw, net)-free.
\end{conjecture}

Clearly if a graph is strongly $e$-positive, then it is (claw, net)-free since both the claw and net are not $e$-positive.  However, proving the other direction seems to be quite challenging. All (claw, net)-free graphs up to and including 9 vertices were verified by computer to be $e$-positive. This provides strong evidence in support of the conjecture.  

The contrapositive of part of this conjecture is:
\begin{conjecture}
\label{mainconj}
If $G$ is not $e$-positive, then $G$ contains an induced claw or an induced net.
\end{conjecture}

Note a kinship between strongly $e$-positive and the {\em nice} property of Stanley \cite{stanleygarsiatribute} where a graph $G$ is nice if whenever there is a stable partition of $G$ of type $\lambda$ (i.e.\  a partition into stable sets of size $\lambda_1, \lambda_2, \ldots$ ) and whenever $\mu \leq \lambda$ in dominance order, there exists a stable partition of type $\mu$.  Then Proposition 1.6 of \cite{stanleygarsiatribute} states that a graph $G$ and all its induced subgraphs are nice if and only if $G$ is claw-free.


\begin{center}
{\bf Acknowledgement}
\end{center}
This work was supported by the Canadian Tri-Council Research
Support Fund. The authors A.M.H. and C.T.H. were each supported by
individual NSERC Discovery Grants.  This research was enabled in part by support provided by Compute Ontario (computeontario.ca) and Compute Canada (www.computecanada.ca). This work was done by author O.D.M. in partial fulfillment of the course requirements for CP493 and CP494: Directed Research Project I  and II in the Department of Physics and Computer Science at Wilfrid Laurier University.

\end{document}

%% file: All4VertexGraphs.tex
\begin{center}
\begin{tikzpicture} [scale = 1.25]
\tikzstyle{every node}=[font=\small]

\newcommand{\size}{1}

\newcommand{\p4}{1}{
    \path (\size * 4, 0) coordinate (g1);
    \path (g1) +(-\size, 0) node (g1_1){};
    \path (g1) +(0, 0) node (g1_2){};
    \path (g1) +(\size, 0) node (g1_3){};
    \path (g1) +(2 * \size, 0) node (g1_4){};
    \foreach \Point in {(g1_1), (g1_2), (g1_3), (g1_4)}{
        \node at \Point{\textbullet};
    }
    \draw   (g1_1) -- (g1_2)
            (g1_2) -- (g1_3)
            (g1_3) -- (g1_4);
    \path (g1) ++(\size  / 2,-\size / 2) node[draw=none,fill=none] { {\large $P_4$}};
}

\newcommand{\kfour}{2}{
    \path (0, - \size * 2) coordinate (g2);
    \path (g2) +(0, 0) node (g2_1){};
    \path (g2) +(0, \size) node (g2_2){};
    \path (g2) +(\size, \size) node (g2_3){};
    \path (g2) +(\size, 0) node (g2_4){};
    \foreach \Point in {(g2_1), (g2_2), (g2_3), (g2_4)}{
        \node at \Point{\textbullet};
    }
    \draw   (g2_1) -- (g2_2)
            (g2_1) -- (g2_3)
            (g2_1) -- (g2_4)
            (g2_2) -- (g2_3)
            (g2_2) -- (g2_4)
            (g2_3) -- (g2_4);
    \path (g2) ++(\size  / 2,-\size / 2) node[draw=none,fill=none] { {\large $K_4$}};
}

\newcommand{\diam}{3}{
    \path (\size * 2, - \size * 2) coordinate (g3);
    \path (g3) +(0, 0) node (g3_1){};
    \path (g3) +(0, \size) node (g3_2){};
    \path (g3) +(\size, \size) node (g3_3){};
    \path (g3) +(\size, 0) node (g3_4){};
    \foreach \Point in {(g3_1), (g3_2), (g3_3), (g3_4)}{
        \node at \Point{\textbullet};
    }
    \draw   (g3_1) -- (g3_2)
            (g3_1) -- (g3_3)
            (g3_1) -- (g3_4)
            (g3_2) -- (g3_3)
            (g3_3) -- (g3_4);
    \path (g3) ++(\size  / 2,-\size / 2) node[draw=none,fill=none] { {\large $diamond$}};
}

\newcommand{\cfour}{4}{
    \path (\size * 4, - \size * 2) coordinate (g4);
    \path (g4) +(0, 0) node (g4_1){};
    \path (g4) +(0, \size) node (g4_2){};
    \path (g4) +(\size, \size) node (g4_3){};
    \path (g4) +(\size, 0) node (g4_4){};
    \foreach \Point in {(g4_1), (g4_2), (g4_3), (g4_4)}{
        \node at \Point{\textbullet};
    }
    \draw   (g4_1) -- (g4_2)
            (g4_1) -- (g4_4)
            (g4_2) -- (g4_3)
            (g4_3) -- (g4_4);
    \path (g4) ++(\size  / 2,-\size / 2) node[draw=none,fill=none] { {\large $C_4$}};
}

\newcommand{\paw}{5}{
    \path (\size * 6, - \size * 2) coordinate (g5);
    \path (g5) +(0, 0) node (g5_1){};
    \path (g5) +(0, \size) node (g5_2){};
    \path (g5) +(\size, \size) node (g5_3){};
    \path (g5) +(\size, 0) node (g5_4){};
    \foreach \Point in {(g5_1), (g5_2), (g5_3), (g5_4)}{
        \node at \Point{\textbullet};
    }
    \draw   (g5_1) -- (g5_2)
            (g5_2) -- (g5_3)
            (g5_2) -- (g5_4)
            (g5_3) -- (g5_4);
    \path (g5) ++(\size  / 2,-\size / 2) node[draw=none,fill=none] { {\large $paw$}};
}

\newcommand{\claw}{6}{
    \path (\size * 8, - \size * 2) coordinate (g6);
    \path (g6) +(0, 0) node (g6_1){};
    \path (g6) +(0, \size) node (g6_2){};
    \path (g6) +(\size, \size) node (g6_3){};
    \path (g6) +(\size, 0) node (g6_4){};
    \foreach \Point in {(g6_1), (g6_2), (g6_3), (g6_4)}{
        \node at \Point{\textbullet};
    }
    \draw   (g6_1) -- (g6_2)
            (g6_2) -- (g6_4)
            (g6_2) -- (g6_3);
    \path (g6) ++(\size  / 2,-\size / 2) node[draw=none,fill=none] { {\large $claw$}};
}

\newcommand{\cokfour}{7}{
    \path (0, - \size * 4) coordinate (g7);
    \path (g7) +(0, 0) node (g7_1){};
    \path (g7) +(0, \size) node (g7_2){};
    \path (g7) +(\size, \size) node (g7_3){};
    \path (g7) +(\size, 0) node (g7_4){};
    \foreach \Point in {(g7_1), (g7_2), (g7_3), (g7_4)}{
        \node at \Point{\textbullet};
    }

    \path (g7) ++(\size  / 2,-\size / 2) node[draw=none,fill=none] { {\large $4K_1$}};
}

\newcommand{\codiamond}{8}{
    \path (\size * 2, - \size * 4) coordinate (g8);
    \path (g8) +(0, 0) node (g8_1){};
    \path (g8) +(0, \size) node (g8_2){};
    \path (g8) +(\size, \size) node (g8_3){};
    \path (g8) +(\size, 0) node (g8_4){};
    \foreach \Point in {(g8_1), (g8_2), (g8_3), (g8_4)}{
        \node at \Point{\textbullet};
    }
    \draw   (g8_3) -- (g8_4);
    \path (g8) ++(\size  / 2,-\size / 2) node[draw=none,fill=none] { {\large co-$diamond$}};
}

\newcommand{\cocfour}{9}{
    \path (\size * 4, - \size * 4) coordinate (g8);
    \path (g8) +(0, 0) node (g8_1){};
    \path (g8) +(0, \size) node (g8_2){};
    \path (g8) +(\size, \size) node (g8_3){};
    \path (g8) +(\size, 0) node (g8_4){};
    \foreach \Point in {(g8_1), (g8_2), (g8_3), (g8_4)}{
        \node at \Point{\textbullet};
    }
    \draw (g8_1) -- (g8_2)
          (g8_3) -- (g8_4);
    \path (g8) ++(\size  / 2,-\size / 2) node[draw=none,fill=none] { {\large $2K_2$}};
}

\newcommand{\copaw}{10}{
    \path (\size * 6, - \size * 4) coordinate (g9);
    \path (g9) +(0, 0) node (g9_1){};
    \path (g9) +(0, \size) node (g9_2){};
    \path (g9) +(\size, \size) node (g9_3){};
    \path (g9) +(\size, 0) node (g9_4){};
    \foreach \Point in {(g9_1), (g9_2), (g9_3), (g9_4)}{
        \node at \Point{\textbullet};
    }
    \draw   (g9_1) -- (g9_2)
            (g9_2) -- (g9_3);
    \path (g9) ++(\size  / 2,-\size / 2) node[draw=none,fill=none] { {\large co-$paw$}};
}

\newcommand{\coclaw}{11}{
    \path (\size * 8, - \size * 4) coordinate (g10);
    \path (g10) +(0, 0) node (g10_1){};
    \path (g10) +(0, \size) node (g10_2){};
    \path (g10) +(\size, \size) node (g10_3){};
    \path (g10) +(\size, 0) node (g10_4){};
    \foreach \Point in {(g10_1), (g10_2), (g10_3), (g10_4)}{
        \node at \Point{\textbullet};
    }
    \draw   (g10_1) -- (g10_2)
            (g10_1) -- (g10_3)
            (g10_2) -- (g10_3);
    \path (g10) ++(\size  / 2,-\size / 2) node[draw=none,fill=none] { {\large co-$claw$}};
}

\end{tikzpicture}
\end{center}